\documentclass[12pt,reqno]{amsart}
\usepackage{amsmath,amssymb,amsfonts,amsthm}
\usepackage[left=3cm, right=3cm, top=2.8cm, bottom=2.8cm]{geometry}
\usepackage[utf8]{inputenc}
\usepackage[T1]{fontenc}
\usepackage{color}
\usepackage{graphicx}
\usepackage{stackrel}

\newcommand\supp{\mathop{\rm supp}}

\renewenvironment{proof}[1][\proofname]{{\itshape\bfseries \hspace*{-0.58cm} #1\hspace*{0,05cm}.\hspace*{0,2cm}}}{\qed\vspace{0,3cm}}

\newtheorem{theorem}{Theorem}
\newtheorem{lemma}[theorem]{Lemma}
\newtheorem{proposition}[theorem]{Proposition}
\newtheorem{corollary}[theorem]{Corollary}
\newtheorem{definition}[theorem]{Definition}
\newtheorem{remark}[theorem]{Remark}

\newtheorem*{remark*}{Remark}
\newtheorem*{theorem*}{Theorem}

\begin{document}

\title[Conditions for the absence of blowing up solutions]{Conditions for the absence of blowing up solutions to fractional differential equations}


\author[P. M. Carvalho-Neto]{Paulo M. de Carvalho-Neto}
\address[Paulo M. de Carvalho Neto]{Departamento de Matem\'atica, Centro de Ciências Físicas e Matemáticas, Universidade Federal de Santa Catarina, Florian\'{o}polis - SC, Brazil}
\email[]{paulo.carvalho@ufsc.br}
\author[R. Fehlberg Junior]{Renato Fehlberg Junior}
\address[Renato Fehlberg Junior]{Departamento de Matem\'atica, Universidade Federal do Esp\'{i}rito Santo, Vit\'{o}ria - ES, Brazil}
\email[]{fjrenato@yahoo.com.br}


\begin{abstract}
  When addressing ordinary differential equations in infinite dimensional Banach spaces, an interesting question that arises concerns the existence (or non existence) of blowing up solutions in finite time. In this manuscript we discuss this question for the fractional differential equation $cD_t^\alpha u=f(t,u)$ proving that when $f$ is locally Lipschitz, however does not maps bounded sets into bounded sets, we can construct a maximal local solution that not ``blows up'' in finite time.
\end{abstract}

\maketitle

\section{Introduction}

To introduce the main aspects of this discussion, consider the following differential equation
\begin{equation}\label{ode1}\left\{\begin{array}{ll}u^\prime(t)=f(t,u(t)),&t>t_0,\vspace*{0.2cm}\\
u(t_0)=u_0\in X,\end{array}\right.\end{equation}
where $X$ denotes a Banach space and $t_0\in\mathbb{R}$.

Conditions for the existence of solutions to problem \eqref{ode1} were firstly obtained by Peano in \cite{Pe1}. Following his work, several mathematicians proposed improvements to this result. Nonetheless, to our objectives in this manuscript it worths to present the following version (see \cite{Dei1,Hal1} for details of the proof).
\begin{theorem}\label{odetheo} Let $X$ be a Banach space, $t_0\in\mathbb{R}$ and $f:[t_0,\infty)\times X\rightarrow X$ a continuous function which is locally Lipschitz and bounded. Then problem \eqref{ode1} has a global solution in the interval $[t_0,\infty)$ or there exists a value $\omega\in (t_0,\infty)$ such that $u:[t_0,\omega)\rightarrow X$ is a maximal local solution that satisfies
$$\limsup_{t\rightarrow\omega^{-}}\|u(t)\|_X=\infty.$$
\end{theorem}

Clearly the above result establishes conditions such that the solutions of \eqref{ode1} have a dichotomy property concerning its ``longtime behavior''. Since the need of the Lipschitz condition is classical, it is natural that questions regarding the necessity of the boundedness property of $f$ were raised.

If we focus on finite dimensional spaces, it seems obviously that this hypothesis is dispensable, however in infinite dimensional spaces it plays a fundamental role. Dieudonn\'{e} in \cite{Di1} was most likely the first mathematician to address this question. He considered the Banach space
$$X:=\big\{\{x_n\}_{n=1}^\infty:x_n\in\mathbb{\mathbb{R}}\textrm{ and }\lim_{n\rightarrow\infty}{x_n}=0\big\},$$
 with norm $\|\{x_n\}_{n=1}^\infty\|_X:=\sup_{n\in\mathbb{N}}|x_n|$ and construct a non-bounded and Lipschitz function $f:[0,1]\times X\rightarrow X$ such that $(1)$ posses a local solution that does not admit a continuation and is also bounded.

Following Dieudonn\'{e}'s inspiration, many mathematicians discussed this kind of problem, which in certain sense, is related with the failure of Peano's existence theorem in infinite dimensions. For instance, Deimling improved Dieudonn\'{e} construction considering more general Banach spaces, as can be seen in \cite{Dei1,Dei2}; Komornik et al. in \cite{KoMaPiVa1} addressed the autonomous version of \eqref{ode1}, proving that for any infinite dimensional Banach space $X$ and bounded interval $(s,t)\subset\mathbb{R}$, there exists a locally Lipschitz function $f:X\rightarrow X$ and $u_0\in X$ such that the maximal solution of \eqref{ode1} is exactly defined on $(s,t)$ although it remains bounded on $(s,t)$.

On the other hand, it is worth recalling that fractional differential equations are gaining considerable emphasis in the mathematical society and the equivalent question in this context, besides being a very interesting problem, was still unanswered.

Even if this question seems to have an adaptable proof from the standard case of ordinary differential equations, it does not happen. The non-local characteristic carried by the fractional differential operator is very hard to be manipulated and new arguments are necessary to obtain such result.

In order to fill this gap, we initially recall some results of the fractional differential equations theory and then we exhibit an example of a locally Lipschitz function $f$ that does not map bounded sets onto bounded sets and induces \eqref{ode2}, Fractional Cauchy Problem in Section 2, to possess a maximal local solution that is also bounded.

Finally we present the structure of this paper. In Section \ref{S2} we discuss several important tools concerning the fractional calculus and the respective fractional Cauchy problem, proving the result concerning ``blowing up'' solutions. In Section \ref{S3} we construct the aforementioned counter example, discussing also every functional analysis tools used in the process. We left Section \ref{S4} to address a new perspective to prove Theorem \ref{counterexample}. More specifically, we introduce arguments that allow us to exhibit a constructive proof to Theorem \ref{counterexample}, which, however, have some restrictions.

\section{A Study of Differential Equations with Fractional\\ Caputo Derivative}
\label{S2}

At first let us recall the study of the locally Bochner integrable functions for the Dunford-Schwartz integral with respect to a Banach space (see details in \cite{DuSc1}). Hereafter assume that $S\subset\mathbb{R}$ and $X$ is a Banach space.
 \begin{itemize}
 \item[i)] Denote by $L^1(S,X)$ the set of all measurable functions $x:S\rightarrow X$ such that $\|x(t)\|_X$ is integrable. Furthermore, this set equipped with norm
 $$\|x(t)\|_{L^1(S,X)}:=\displaystyle\int_S{\|x(t)\|_X}\,dt,$$
is a Banach space.

 \item[ii)] Represent by $W^{1,1}(S,X)$ the set of all elements of $L^1(S,X)$ which have weak derivative of order one being in $L^1(S,X)$. This set equipped with norm
 $$\|x(t)\|_{W^{1,1}(S,X)}:= \Bigl(\|x(t)\|^2_{L^1(S,X)}+\|x^\prime(t)\|^2_{L^1(S,X)}\Bigr)^{1/2},$$
is also a Banach space.

 \item[iii)] Finally, $C(S,X)$ denotes the space of the continuous functions $x:S\rightarrow X$. When $S$ is a compact set we define the norm
 $$\|x(t)\|_{C(S,X)}:= \sup_{t\in S}{\|x(t)\|_X},$$
which makes $C(S,X)$ a Banach space.
 \end{itemize}

  \begin{definition}Let $\alpha\in(0,1)$, $\tau\in(0,\infty)$ and $h\in L^1(0,\tau;X)$.
  \begin{itemize}
  \item[i)] The Riemann-Liouville fractional integral of order $\alpha$, is denoted by {$J_t^{\alpha}h(t)$}, and is given by
    $$J_t^{\alpha}h(t):=(h*g_\alpha)(t)=\int_0^t{g_\alpha(t-r)h(r)}\,dr,\qquad \textrm{ a.e. in } [0,\tau],$$
  where function $g_{\alpha}:\mathbb{R}\rightarrow\mathbb{R}$ is given by
\begin{equation*}g_{\alpha}(t)=\left\{\begin{array}{ll}t^{\alpha-1}/\Gamma(\alpha),& t>0,\\
  0,&t\leq0.
  \end{array}\right.\end{equation*}
  \item[ii)] If $h*g_{1-\alpha}\in W^{1,1}(0,\tau;X)$, the { Riemann-Liouville fractional derivative of order $\alpha$} is given by
  $$D_t^{\alpha}h(t):=D_t^{1}J_t^{1-\alpha}h(t)=D_t^1(h*g_{1-\alpha})(t),\qquad \textrm{ a.e. in } [0,\tau],$$
  where $D_t^1=\left(d/dt\right)$.
  \end{itemize}
  \end{definition}

The above definitions were extensively used in the study of fractional calculus (see for instance \cite{LaPeVa1,LiPeJi1,LiLi1,XuJiYa1}) however in this manuscript we discuss only Caputo fractional derivative, which is defined bellow (see for instance \cite{CaPl1,AnCaCaMa1,WaChXi1} for more details).

\begin{definition} Consider real numbers $\alpha\in(0,1)$, $\tau\in(0,\infty)$ and function $h\in C([0,\tau],X)$ satisfying $h*g_{1-\alpha}\in W^{1,1}(0,\tau;X)$. We define the {Caputo fractional derivative of order $\alpha$}, which is denoted by {$cD_t^{\alpha}h(t)$}, by
  $$cD_t^{\alpha}h(t):=D_t^{\alpha}\big(h(t)-h(0)\big),\qquad \textrm{ a.e. in } [0,\tau].$$
  \end{definition}

\begin{remark} \label{prop} It is important to emphasize some facts at this moment.
\begin{itemize}
\item[i)] In \cite{Car1} the author reproduces the classical proof that ensures, for an integrable function $h\in L^1(0,\tau;X)$, that $D_t^{\alpha}J_t^{\alpha}h(t)=h(t)$. Moreover, if it holds that $h*g_{1-\alpha}$ also belongs to $W^{1,1}(0,\tau;X)$, then
    $$J_t^{\alpha}D_t^{\alpha}h(t)=h(t)-\dfrac{1}{\Gamma(\alpha)}t^{\alpha-1}\left\{J_s^{1-\alpha}h(s)\right\}\big|_{s=0},\qquad \textrm{ a.e. in } [0,\tau].$$
    It is important to notice that $\left\{J_s^{1-\alpha}h(s)\right\}\big|_{s=0}$ can not even be computable; choose for instance $h(t)=t^{-1/2}$.
\item[ii)]For $h\in C([0,\tau],X)$, it holds that
    $$cD_t^{\alpha}J_t^{\alpha}h(t)=h(t).$$
    Furthermore, if $h*g_{1-\alpha}\in W^{1,1}(0,\tau;X)$ we conclude
    $$J_t^{\alpha}cD_t^{\alpha}h(t)=
    h(t)-h(0)-\dfrac{1}{\Gamma(\alpha)}t^{\alpha-1}\left\{J_s^{1-\alpha}[h(s)-h(0)]\right\}\big|_{s=0},$$
    a.e. in $[0,\tau]$. Since $\left\{J_s^{1-\alpha}[h(s)-h(0)]\right\}\big|_{s=0}=0$, we achieve the equality
    $$J_t^{\alpha}cD_t^{\alpha}h(t)=h(t)-h(0).$$
\item[iii)] Finally, if $u\in C([0,\tau],X)$ is a function that satisfies $u^\prime\in C([0,\tau],X)$ we obtain $cD_t^\alpha u(t)=J_{t}^{1-\alpha}u^\prime(t)$ for any $t\in(0,\tau]$. This was the first formal definition of the Caputo fractional derivative.
\end{itemize}
\end{remark}

Once the main tools concerning fractional calculus are introduced, we now formalize the fractional differential equation that we study in this manuscript. The fractional Cauchy problem (FCP) is given by
\begin{equation}\tag{FCP}\label{ode2}\left\{\begin{array}{ll}cD_t^\alpha u(t)=f(t,u(t)),&t>0,\\
u(0)=u_0\in X,\end{array}\right.\end{equation}
where $\alpha$ is a real number in $(0,1)$, $cD_t^\alpha$ is the Caputo fractional derivative of order $\alpha$ and $f:[0,\infty)\times X\rightarrow X$ is a continuous function.

Now it is necessary to define the notion of solution to problem \eqref{ode2}, which indeed is given by an adaptation of the classical ideas that are applied to ordinary differential equations.

\begin{definition}\label{defisol} Assume that $\alpha\in(0,1)$.
\begin{itemize}
\item[$i)$] We say that a function $u:[0,\infty)\rightarrow X$ is a global solution of \eqref{ode2} if
\begin{equation*}u\in C^{\alpha}([0,\tau],X):=\{u\in C([0,\tau],X):cD_t^\alpha u\in C([0,\tau],X)\}\end{equation*}
for every $\tau>0$ and satisfies the equations of \eqref{ode2}.\vspace{0.2cm}
\item[$ii)$] If there exists $0<\tau<\infty$ such that a continuous function $u:[0,\tau]\rightarrow X$ belongs to $C^{\alpha}([0,\tau],X)$ and satisfies \eqref{ode2} for $t\in[0,\tau]$, we say that $u$ is a local solution to problem \eqref{ode2} on the interval $[0,\tau]$.
\end{itemize}
\end{definition}

\begin{remark} Observe that, viewed as a subspace of $C([0,\tau],X)$, the space $C^{\alpha}([0,\tau],X)$ is a Banach space.
\end{remark}

Bearing these definitions in mind, we present the classical result that discuss the local existence and uniqueness of a solution to the fractional differential equation (a proof of this theorem can be found in \cite{Car1,KiSrTr1}).

\begin{theorem}\label{exisuni1} Assume that $\alpha\in(0,1)$, $f:[0,\infty)\times X\rightarrow X$ is a continuous function and $u_0\in X$. If $f$ is also a locally Lipschitz function, i.e., given $(t_0,x_0)\in[0,\infty)\times X$ there exist $L,r>0$ (depending on $f$, $t_0$ and $x_0$) such that for any $(t,x),(t,y)\in B_{r}(t_0,x_0)$ it holds that
$$\|f(t,x)-f(t,y)\|_X\leq L\|x-y\|_X,$$
then there exists $\tau>0$ such that problem \eqref{ode2} posses a unique local solution $u$ in $[0,\tau]$.
\end{theorem}

The remainder of this section will be devoted to discuss the continuation of local solutions and global solutions of \eqref{ode2}. First, it is necessary to introduce some concepts.

\begin{definition}\label{3.31} Let $u:[0,{\tau}]\rightarrow X$ be a local solution to \eqref{ode2}.
\begin{itemize}
\item[i)] If ${\tau^*}>{\tau}$ and $u^*:[0,{\tau^*}]\rightarrow X$ is a local solution to \eqref{ode2} in $[0,{\tau^*}]$ such that $u(t)=u^*(t)$ in $[0,\tau]$, then we call $u^*$ a continuation of $u$ over $[0,{\tau^*}]$.
\item[ii)] Furthermore, if $u:[0,{{\tau^*}})\rightarrow X$ is the unique local solution to \eqref{ode2} in $[0,{\tau}]$ for every $\tau\in(0,\tau^*)$ and does not have a continuation, then we call it maximal local solution of \eqref{ode2} in $[0,\tau^*)$ (see \cite{St1} for more details on maximal solutions).
\end{itemize}
\end{definition}

Now we are able to establish the existence of continuation to a given solution of \eqref{ode2}.

\begin{theorem} \label{3.4} Let $\alpha\in(0,1)$, $\tau\in(0,\infty)$ and $f:[0,\infty)\times X\rightarrow X$ be as in Theorem \ref{exisuni1}. If $u:[0,\tau]\rightarrow X$ is the unique local solution to \eqref{ode2} in $[0,\tau]$, then there exists a unique continuation $u^*$ of $u$ in $[0,\tau^*]$ for some value $\tau^*>\tau.$
\end{theorem}

Finally, inspired by De Andrade et al. in \cite{AnCaCaMa1} and based on the results discussed above, for the sake of completeness we state and prove the main theorem of this section.

\begin{theorem}\label{odefractheo} Let $\alpha\in(0,1)$, $X$ be a Banach space and $f:[0,\infty)\times X\rightarrow X$ a continuous function which is locally Lipschitz and maps bounded sets onto bounded sets. Then problem \eqref{ode2} has a global solution in the interval $[0,\infty)$ or there exists a value $\omega\in (0,\infty)$ such that the local solution $u:[0,\omega)\rightarrow X$ does not admit a continuation and yet satisfies
$$\limsup_{t\rightarrow\omega^{-}}\|u(t)\|_X=\infty.$$
\end{theorem}

\begin{proof} Consider $H\subset\mathbb{R}$, which is given by
\begin{multline}H:=\big\{\tau\in(0,\infty): \textrm{ there exists } u_\tau:[0,\tau]\rightarrow X\,\\ \text{unique local solution to \eqref{ode2} in } [0,\tau]\big\}.\end{multline}

Define $w=\sup{H}$ and consider function $u:[0,\omega)\rightarrow X$ which is given by $u(t)=u_\tau(t)$, if $t\in[0,\tau]$. It is not difficult to verify that this function is well defined and is the maximal local solution of \eqref{ode2} in $[0,\omega)$.

If $\omega=\infty$, $u$ is a global solution of \eqref{ode2}. Otherwise, if $\omega<\infty$ we need to prove that %
$$\limsup_{t\rightarrow\omega^-}{\|u(t)\|_X}=\infty.$$

The proof is by contradiction. Suppose that there exists $d<\infty$ such that $\|u(t)\|_X\leq d$ for all $t\in[0,\omega)$. Then, since $f$ maps bounded sets onto bounded sets, define
$$M:=\sup_{s\in[0,\omega)}{\|f(s,u(s))\|_X}<\infty$$
and consider $\{t_n\}_n\subset[0,\omega)$ a sequence that converges to $\omega$. Thus, making some computations, we obtain the estimate
$$\|u(t_n)-u(t_m)\|_X\leq \dfrac{M^*}{\Gamma(\alpha+1)}\Big[|{t_n}^\alpha+|t_m-t_n|^\alpha-t_m^\alpha| + |t_m-t_n|^\alpha\Big],$$
for some positive value $M^*$. This ensures that $\{u(t_n)\}_{n=0}^{\infty}$ is a Cauchy sequence and therefore it has a limit, let us say, $u_\omega\in X$. By extending $u$ over $[0,\omega]$, we conclude that the equality
$$u(t)=u_{0}+\int_{0}^{t}{(t-s)^{\alpha-1}f(s,u(s))}\,ds,$$
should hold for all $t\in [0,\omega]$. With this, by Theorem \ref{3.4}, we can extend the solution to some bigger interval, which is a contradiction by the definition of $\omega$. Therefore, if $\omega<\infty$ it holds that $\limsup_{t\rightarrow\omega^-}{\|u(t)\|_X}=\infty.$ This concludes the proof.

\end{proof}

\section{Fundamental Structures and the Bounded Maximal Solution}
\label{S3}
The aim of this section is to recall some fundamental concepts of functional analysis and discuss the existence of a maximal local solution to problem $\eqref{ode2}$ which does not ``blows up'' in finite time, under suitable hypotheses.

For the discussion suggested above to be successfully addressed, we first recall some concepts and notations related to infinite dimensional Banach spaces.

\begin{definition}\label{sch} Let $X$ be a Banach space. A sequence $\{v_n\}_{n=1}^{\infty}\subset X$ is called a Schauder basis of $X$, if for every $x\in X$, there exists a unique sequence $\{x_n\}_{n=1}^{\infty}\subset\mathbb{R}$ such that
\begin{equation}\label{schauder}\lim_{k\rightarrow\infty}{\|x-\sum_{n=1}^{k}x_nv_n\|_X}=0.\end{equation}
We write $x=\displaystyle\sum_{n=1}^{\infty}x_nv_n$ to denote the above limit.
\end{definition}

It worths to emphasize that existence of Schauder basis to general Banach spaces is not a trivial matter. Indeed, as can be found in the literature, it is not true that every Banach space has a Schauder basis (see \cite{En1} for details).

Therefore, it is essential to introduce the following result to better adjust our forward computations.

\begin{theorem}\label{3.660} For any infinite dimensional Banach space $X$, there exists an infinite dimensional closed subspace $X_0$ of $X$ with a Schauder Basis $\{v_n\}_{n=0}^{\infty}$. Moreover, we can suppose that $\{v_n\}_{n=1}^{\infty}$ is such that $\|v_n\|_X=1$, for all $n\in\mathbb{N}$, and that there exists a sequence of linear functionals $\{{v_n}^*\}_{n=1}^{\infty}\subset {X_0}^*$ which satisfies $\|v_n\|_{X_0^*}=1$ and also, for any $x\in X_0$,
\begin{equation}\label{3.662}x=\sum_{n=1}^{\infty}{v_n}^*(x) v_n.
\end{equation}
\end{theorem}
\begin{proof} It is a classical result. For more details see \cite[Theorem I.1.2]{LiTz1}.
\end{proof}

At this point we are already prepared to discuss the main ideas proposed by this manuscript. Thus, the remainder of this section is dedicated to prove the following:\\

\noindent{\bf Statement:} There exists an element $u_0\in X$ and also a continuous and locally Lipschitz function $f:\mathbb{R}^+\times X\rightarrow X$, which does not maps bounded sets into bounded sets, such that
\begin{equation}\tag{FCP}\left\{\begin{array}{ll}cD_t^\alpha u(t)=f(t,u(t)),&t>0\\
u(0)=u_0\in X.\end{array}\right.\end{equation}
has a maximal bounded local solution in $[0,1)$.\\

A positive answer to the problem above closes any question concerning the adopted hypotheses in Theorem \ref{odefractheo}. It is worth to stress that the ideas which inspired the proof of this question were given by Dieudonn\'{e} in \cite{Di1}, Deimling in \cite{Dei1,Dei2} and Komornik in \cite{KoMaPiVa1}.

\begin{proposition}\label{goodfunction} Given real numbers $s_1<s_2<s_3<s_4$, there exists a continuously differentiable function $z:\mathbb{R}\rightarrow\mathbb{R}$ satisfying $\supp(z)\subset(s_1,s_4)$, $z(t)\equiv1$ for $t\in[s_2,s_3]$, $z^\prime(t)\geq0$ for $t\in[s_1,s_2]$ and $z^\prime(t)\leq0$ for $t\in[s_3,s_4]$.\end{proposition}

\begin{proof} Initially consider any real numbers $s_1<s_2<s_3<s_4$ and the continuously differentiable function $\theta:\mathbb{R}\rightarrow\mathbb{R}$ given by
$$ \theta(t):=\left\{\begin{array}{cl}e^{-1/t},&\textrm{if }t>0,\vspace*{0.2cm}\\0,&\textrm{if }t\leq0.\end{array}\right.$$

Then define functions $\mu(t):=\theta(s_2-t)\theta(t-s_1)$ and $\nu(t):=\theta(s_4-t)\theta(t-s_3)$ to finally obtain
$$ \eta(t):=\left(\displaystyle\int_{-\infty}^t{\mu(s)}\,ds\right)\left(\displaystyle\int^{\infty}_t{\nu(s)}\,ds\right), \textrm{ for }t\in\mathbb{R}.$$

It is not difficult to conclude that $z(t):=\eta(t)/\eta$, where
$$ \eta:=\left(\displaystyle\int_{s_1}^{s_2}{\mu(s)}\,ds\right)\left(\displaystyle\int^{s_4}_{s_3}{\nu(s)}\,ds\right),$$
is such that $z\in C^1(\mathbb{R};\mathbb{R})$, $\supp(z)\subset(s_1,s_4)$ and $z(t)\equiv1$, for $t\in[s_2,s_3]$. The conclusions concerning the sign of the derivative are trivial.

\end{proof}

Last proposition allow us to construct a suitable sequence of functions, which develops a fundamental role forward in the manuscript.

\begin{corollary}\label{goodfunction2} There exist an increasing sequence of positive real numbers $\{t_n\}_{n=1}^{\infty}$ that converges to $1$ and a sequence of continuously differentiable functions $\{z_n(t)\}_{n=1}^{\infty}$, which satisfies the following properties:
\begin{itemize}
\item[i)] The length of the intervals $[t_n,t_{n+1}]$ decreases when $n$ increases. More specifically,
$$t_n-t_{n-1}>t_{n+1}-t_n,$$
for $n\geq2$;\vspace{0.2cm}
\item[ii)] $z_1\equiv 1$ in $\mathbb{R}$ and to $n\geq2$
$$z_n(t):=\left\{\begin{array}{ll}1,&\quad t\in[t_n,t_{n+1}],\\\\
\in(0,1),&\quad t\in\left(\dfrac{t_{n-1}+t_n}{2},t_n\right)\cup\left(t_{n+1},\dfrac{t_{n+1}+t_{n+2}}{2}\right),\\\\
0,&\quad \textrm{otherwise}.
\end{array}\right.$$
\end{itemize}
\end{corollary}

\begin{proof} It is not difficult to notice that this is a consequence of Proposition \ref{goodfunction}.
\end{proof}

Now let us start by considering $X$ an infinite dimensional Banach space. As stated in Theorem \ref{3.660}, we assume that $X_0$ denotes an infinite dimensional closed subspace of $X$ (with the topology induced by the topology of $X$) with a Schauder Basis $\{v_n\}_{n=1}^{\infty}$, which satisfies
$$\|v_n\|_X=1,\quad \forall n\in\mathbb{N}.$$
We also recall that there exists a sequence of linear functionals $\{{v_n}^*\}_{n=1}^{\infty}\subset {X_0}^*$ which satisfies
$$\|v_n^*\|_{X_0^*}=\stackrel[\|x\|_X\leq1]{}{\sup_{x\in X_0}}{|v_n^*(x)|}=1,\quad \forall n\in\mathbb{N},$$
and allow us to write every $x\in X_0$ as the sum
$$x=\sum_{n=1}^{\infty}{v_n}^*(x) v_n.$$

The characteristic function $\chi_{I}:\mathbb{R}\rightarrow\mathbb{R}$ is the function given by
$$\chi_{I}(t)=\left\{\begin{array}{ll}1,&t\in I,\\0,&t\notin I.
\end{array}\right.$$

Based on this last considerations, we prove the following result.

\begin{lemma}\label{solution} The function $u:[0,1)\rightarrow X_0\, (\subset X)$ given by
$$u(t):=\sum_{n=1}^{\infty}z_{n}(t)v_n$$
is continuous and bounded. Moreover, it cannot be extended to $[0,1]$ and if $\alpha\in(0,1)$, it satisfies
$$cD_t^\alpha u(t)=\sum_{n=1}^\infty{\chi_{[t_n,1)}(t)cD_t^\alpha z_{n+1}(t)v_{n+1}}$$
with $cD_t^\alpha u(t)$ a continuous function.
\end{lemma}

\begin{proof} The continuity and the boundedness of $u(t)$ follows since $\|z_n(t)v_n\|_X\leq 1$ for every $n\in\mathbb{N}$ and
\begin{equation}\label{solution01}u(t)=\left\{\begin{array}{ll}v_1,&\textrm{ for }t\in[0,t_1)\vspace*{0.1cm}\\
v_1+z_2(t)v_2,&\textrm{ for }t\in[t_1,t_2)\vspace*{0.1cm}\\
v_1+z_2(t)v_2+z_3(t)v_3,&\textrm{ for }t\in[t_2,t_3)\vspace*{0.1cm}\\
v_1+z_2(t)v_2+z_3(t)v_3+z_4(t)v_4,&\textrm{ for }t\in[t_3,t_4)\vspace*{0.1cm}\\
v_1+z_{k-1}(t)v_{k-1}+z_{k}(t)v_{k}+z_{k+1}(t)v_{k+1},&\textrm{ for }t\in[t_k,t_{k+1})\\
&\textrm{ and }k\geq4.
\end{array}\right.\end{equation}

To verify that $u$ cannot be extended in $[0,1]$ continuously, observe that for $n\geq2$
$$u\left(\dfrac{t_n+t_{n+1}}{2}\right)=v_1+v_{n}.$$
Thus, define $\sigma_n=(t_n+t_{n+1})/2$ and observe that
$$\|u(\sigma_n)-u(\sigma_{n+1})\|_X=\|v_n-v_{n+1}\|_X.$$
Since $\{v_n\}_{n=1}^\infty$ cannot be convergent, the sequence $u\left(\dfrac{t_n+t_{n+1}}{2}\right)$ cannot be a Cauchy sequence, and therefore does not converges. In other words, there is no way to define a value at $t=1$ such that function $u$ becomes continuous in $[0,1]$.

By recalling the already mentioned non-local property of the fractional derivative and applying $cD_t^\alpha$ in \eqref{solution01} we obtain
\begin{equation*}cD_t^\alpha u(t)=\left\{\begin{array}{ll}0,&\textrm{ for }t\in[0,t_1)\vspace*{0.1cm}\\
cD_t^\alpha z_2(t)v_2,&\textrm{ for }t\in[t_1,t_2)\vspace*{0.1cm}\\
cD_t^\alpha z_2(t)v_2+cD_t^\alpha z_3(t)v_3,&\textrm{ for }t\in[t_2,t_3)\vspace*{0.1cm}\\
cD_t^\alpha z_2(t)v_2+cD_t^\alpha z_3(t)v_3+cD_t^\alpha z_4(t)v_4,&\textrm{ for }t\in[t_3,t_4)\vspace*{0.1cm}\\
\displaystyle\sum_{i=2}^{k+1}{cD_t^\alpha z_i(t)v_i},&\textrm{ for }t\in[t_k,t_{k+1})\vspace*{-0,3cm}\\
&\textrm{ and }k\geq4.
\end{array}\right.\end{equation*}

The continuity of $cD_t^\alpha u(t)$ follows from the fact that $z_{k}(t)$ has its first derivative continuous and therefore Remark \ref{prop} allows us to conclude that $cD_t^\alpha z_{k}(t)=J_t^{1-\alpha}z_k^\prime(t)$.
\end{proof}

At this point, it remains for us to address the construction of the function $f_\alpha:\mathbb{R}^+\times X\rightarrow X$. Thus, consider the following preliminary result.

\begin{lemma}\label{funcaoH} Let function $H:[0,\infty)\times X\rightarrow \mathbb{R}$ be given by
\begin{multline*}H(t,x)={V^*_{1}}(x)\chi_{[0,t_1)}(t)\\+\sum_{i=1}^{\infty}\left[i^{-2}{V^*_{i}}(x)\chi_{[t_i,\infty)}(t)+(i+1)^{-2}{V^*_{i+1}}(x)\left(\dfrac{t-t_{i}}{t_{i+1}-t_i}\right)\chi_{[t_i,t_{i+1})}(t)\right],
\end{multline*}
where $V_i^*:X\rightarrow\mathbb{R}$ is an extension of $v_i^*:X_0\rightarrow\mathbb{R}$ over $X$ and satisfies $\|V_i^*\|_{X^*}\leq1$. Then $H$ is a continuous and locally Lipschitz function. \end{lemma}
\begin{proof} Recall that Hahn-Banach theorem allows us to extend our functionals $v_i^*$ to $V_i^*$ defined over $X$ such that $\|V_i^*\|_{X^*}\leq\|v_i^*\|_{X_0^*}=1$. Then, it is interesting to notice that
\begin{equation*}H(t,x)=\left\{\begin{array}{ll}{V^*_{1}}(x),&\textrm{ for }t\in[0,t_1),\vspace{0.2cm}\\
{V^*_{1}}(x)+\frac{1}{2^2}{V^*_{2}}(x)\left(\dfrac{t-t_{1}}{t_{2}-t_1}\right),&\textrm{ for }t\in[t_1,t_2),\vspace{0.2cm}\\
\displaystyle\sum_{i=1}^{k}{{\frac{1}{i^2}V^*_{i}}(x)}+\frac{1}{(k+1)^2}{V^*_{k+1}}(x)\left(\dfrac{t-t_{k}}{t_{k+1}-t_k}\right),&\textrm{ for }t\in[t_k,t_{k+1})\vspace{-0.2cm}\\
&\textrm{ and }k\geq2,\vspace{0.2cm}\\
\displaystyle\sum_{i=1}^{\infty}{{\frac{1}{i^2}V^*_{i}}(x)},&\textrm{ for }t\in[1,\infty).
\end{array}\right.\end{equation*}

Observe that for $(t,x)\in[0,\infty)\times X$, with $t\neq 1$, the continuity of $H$ follows directly from the above characterization. We verify the continuity of $H$ at $(1,x)$ by definition. Thus, given $\epsilon>0$, choose $k_0\in\mathbb{N}$ such that 
$$\left(\sum_{i=k+2}^\infty{\dfrac{\|x\|_X}{i^2}}\right)+\dfrac{1+\|x\|_X}{(k+1)^2}<\dfrac{\epsilon}{2},$$
for any $k\geq k_0$, and let $\delta_0\in(0,1)$ be such that $1-\delta_0>t_{k_0}$. Then, define
$$\delta:=\min{\left\{\delta_0,\dfrac{\epsilon}{2}\left(\sum_{i=1}^\infty{\dfrac{1}{i^2}}\right)^{-1}\right\}}.$$

If $\|(s,y)-(1,x)\|_{[0,\infty)\times X}<\delta$, we conclude that $s>1-\delta_0>t_{k_0}$ and therefore that $s$ lies in an interval of the form $[t_{k-1},t_k)$, for some $k\geq k_0$ or $s\geq1$. In the first situation, we compute

\begin{equation*}\begin{array}{lll}|H(1,x)-H(s,y)| & = & \left|\displaystyle\sum_{i=1}^{\infty}{{\frac{1}{i^2}V^*_{i}}(x)}-\left[\displaystyle\sum_{i=1}^{k}{{\frac{1}{i^2}V^*_{i}}(y)}\right.\right.\vspace*{0.3cm}\\
&&\hspace*{3cm}\left.\left.+\displaystyle\frac{1}{(k+1)^2}{V^*_{k+1}}(y)\left(\dfrac{s-t_{k}}{t_{k+1}-t_k}\right)\right]\right|\vspace*{0.5cm}\\
& \leq & \left(\displaystyle\sum_{i=k+2}^{\infty}{\dfrac{1}{i^2}}\right)\|x\|_X+\left(\displaystyle\sum_{i=1}^{k}{\dfrac{1}{i^2}}\right)\|x-y\|_X\vspace*{0.1cm}\\
&  &
\hspace*{3cm}+\dfrac{1}{(k+1)^2}\left\|x-y\dfrac{s-t_{k}}{t_{k+1}-t_k}\right\|_X\vspace*{0.5cm}\\
& \leq & \left(\displaystyle\sum_{i=k+2}^{\infty}{\dfrac{1}{i^2}}\right)\|x\|_X+\left(\displaystyle\sum_{i=1}^{k+1}{\dfrac{1}{i^2}}\right)\|x-y\|_X\vspace*{0.1cm}\\
& & \hspace*{3cm}+\dfrac{1}{(k+1)^2}\|y\|_X\vspace{0.5cm}\\
& < & \epsilon,
\end{array}\end{equation*}
since $\|y\|_X\leq \|y-x\|_X+\|x\|_X \leq1+\|x\|_X$ and
$$\left|\left(x-y\dfrac{s-t_{k}}{t_{k+1}-t_k}\right)\right|\leq \|x-y\|+\|y\|,$$
while in the second situation we compute
\begin{equation*}\begin{array}{lll}|H(1,x)-H(s,y)| & = & \left|\displaystyle\sum_{i=1}^{\infty}{{\frac{1}{i^2}V^*_{i}}(x)}-\displaystyle\sum_{i=1}^{\infty}{{\frac{1}{i^2}V^*_{i}}(y)}\right|<\epsilon,
\end{array}\end{equation*}
proving the continuity.

To conclude that $H$ is locally Lipschitz, choose any point $(\widetilde{t},\widetilde{x})$ belonging to $[0,\infty)\times X$.\vspace{0.2cm}

{\itshape 1st Case}: If  $\widetilde{t}\in (0,t_1)$, then consider $2\widetilde{r}:=\min{\{\widetilde{t},t_1-\widetilde{t}\}}$ and observe that given pairs $(t,x),(t,y)\in B_{\widetilde{r}}(\widetilde{t},\widetilde{x})$ we obtain that $t\in(0,t_1)$ and therefore
$$|H(t,x)-H(t,y)|\leq\|V^*_1\|_{X^*}\|x-y\|_X.$$

When $\widetilde{t}=0$, the same inequality holds for any $(t,x),(t,y)\in B_{t_1}(\widetilde{t},\widetilde{x})$ with $t\geq 0$.\vspace{0.2cm}

{\itshape 2nd Case}: If $\widetilde{t}\in (t_k,t_{k+1})$ for some $k\geq 1$, define $2\widetilde{r}:=\min{\{\widetilde{t}-t_k,t_{k+1}-\widetilde{t}\}}$ and notice that for $(t,x),(t,y)\in B_{\widetilde{r}}(\widetilde{t},\widetilde{x})$ we obtain that $t\in(t_k,t_{k+1})$ what ensures
$$H(t,x)-H(t,y)=\left[\sum_{i=1}^{k}{{\frac{1}{i^2}V^*_{i}}(x-y)}\right]+\frac{1}{(k+1)^2}{V^*_{k+1}}(x-y)
\left(\dfrac{t-t_{k}}{t_{k+1}-t_k}\right)$$
therefore,
$$|H(t,x)-H(t,y)|\leq \left[\sum_{i=1}^{k+1}{\frac{1}{i^2}\|V_i^*\|_{X^*}}\right]\|x-y\|_X.$$

{\itshape 3rd Case}:  If $\widetilde{t}=t_k$ for $k\geq 1$, by setting $2\widetilde{r}:=\min{\{\widetilde{t}-t_{k-1},t_{k+1}-\widetilde{t}\}}$ and following the above computations, we achieve the same conclusion.

{\itshape 4th Case}: If $\widetilde{t}\in(1,\infty)$ choose $2\widetilde{r}\in(0,\widetilde{t}-1)$ and observe that for any $(t,x),(t,y)\in B_{\widetilde{r}}(\widetilde{t},\widetilde{x})$ it holds
$$|H(t,x)-H(t,y)|\leq \left[\sum_{i=1}^{\infty}{\frac{1}{i^2}\|V_i^*\|_{X^*}}\right]\|x-y\|_X\leq \left[\sum_{i=1}^{\infty}{\frac{1}{i^2}}\right]\|x-y\|_X.$$

{\itshape 5th Case}: If $\widetilde{t}=1$, choose $2\widetilde{r}\in(0,1)$, and the result follows from the above computations, since the image of points $(t,x)$ and $(t,y)$ would be given by an infinite series or a truncated series. This completes the proof.
\end{proof}

\begin{lemma}\label{funcaof} Consider $\alpha\in(0,1)$ and the function $f_\alpha:\mathbb{R}^+\times X\rightarrow X$ given by
$$f_\alpha(t,x)=\phi(H(t,x))cD_t^\alpha u(t),$$
where $H(t,x)$ is described in Lemma \ref{funcaoH} and $\phi(t):\mathbb{R}\rightarrow\mathbb{R}$ is given by
$$\phi(t)=\min{\{t+1,1\}}.$$
Then $f_\alpha$ is a continuous and locally Lipschitz function. \end{lemma}

\begin{proof} The function $f_\alpha$ is continuous since $\phi(t)$, $cD_t^\alpha u(t)$ and $H(t,x)$ are continuous. As $\phi(t)$ is Lipschitz and $H(t,x)$ is locally Lipschitz by Lemma \ref{funcaoH}, it follows from description (limited number of members in the sum) of $cD_t^\alpha u(t)$ in Lemma \ref{solution} that  $f_\alpha$ is locally Lipschitz.
\end{proof}

Next theorem is the main result of this section, which completely answer the Affirmation.

\begin{theorem}[Sharpness of ``Blow Up'' Conditions]\label{counterexample}Consider $\alpha\in(0,1)$. Then there exists $f_\alpha:\mathbb{R}^+\times X\rightarrow X$ continuous and locally Lipschitz which does not map every bounded set into bounded set, such that problem
  \begin{equation}\tag{FCP}
   \left\{\begin{array}{l}
   cD^\alpha_tu(t)=f_\alpha(t,u(t)),\,\,\,t> 0\\
   u(0)=v_1\in X,
   \end{array} \right.
  \end{equation}
 where $cD_t^\alpha$ is the Caputo's fractional derivative and $v_1$ is the first element of the Schauder basis defined before in this section, posses a bounded maximal solution $u:[0,1)\rightarrow X$.
\end{theorem}
\begin{proof} If we define $f_\alpha$ as in Lemma \ref{funcaof}, then it is continuous and locally Lipschitz. Now, it is not difficult to notice that if $u:[0,1)\rightarrow X$ is given as in Lemma \ref{solution}, we deduce
\begin{equation*}H(t,u(t))=\left\{\begin{array}{ll}z_1(t),&\textrm{ for }t\in[0,t_1),\vspace{0.2cm}\\
z_1(t)+2^{-2}z_2(t)\left(\dfrac{t-t_{1}}{t_{2}-t_1}\right),&\textrm{ for }t\in[t_1,t_2),\vspace{0.2cm}\\
\displaystyle\sum_{i=1}^{k}{i^{-2}z_{i}(t)}+(k+1)^{-2}z_{k+1}(t)\left(\dfrac{t-t_{k}}{t_{k+1}-t_k}\right),&\hspace*{-0,2cm}\textrm{ for }t\in[t_k,t_{k+1}).
\end{array}\right.\end{equation*}
and since $z_k(t)\geq0$, we achieve that $f_\alpha(t,u(t))=cD_t^\alpha u(t)$, which means that $u$ is a solution to problem \eqref{ode2}. Finally $f_\alpha$ does not map bounded sets into bounds sets, since it would contradict Theorem \ref{odefractheo} once $u$ is a bounded maximal solution of \eqref{ode2}.\end{proof}

\section{Final Remarks on Theorem \ref{counterexample}}
\label{S4}

This final section is dedicated to introduce a constructive proof to Theorem \ref{counterexample}. More specifically, we want to exhibit the bounded set which is mapped by $f_\alpha$ into an unbounded set.

To begin this approach, consider $\{\tau_n\}_{n=1}^\infty$, with $\tau_{n}\in[(t_{n-1}+t_n)/2,t_n)$, the maximum value of $z^\prime_{n}(t)$ (which is strictly positive as described by Corollary \ref{goodfunction2}), for each $n\geq2$.

Fix $\alpha\in(0,1)$ and now consider the following hypothesis\vspace{0.2cm}:
\begin{equation*}\textrm{(P)}\left\{\begin{array}{c}\textrm{ There exists } \{\tau_{n_l}\}_{l=1}^\infty\subset\{\tau_{n}\}_{n=1}^\infty \textrm{ such that }\vspace*{0.3cm}\\
\lim_{l\rightarrow\infty}\left\|cD_t^\alpha u(t)\Big|_{t=\tau_{n_l}}\right\|_X=\infty.

\end{array}\right.
\end{equation*}

Observe that hypothesis (P) is not completely artificial as it seems. Recall initially that Proposition \ref{goodfunction} and Corollary \ref{goodfunction2} allow us to verify that the family of functions
$$\{z_n^\prime(t):n\in\mathbb{N}\}$$
should posses an increasing property concerning its maximum in the interval $[(t_{n-1}+t_n)/2,t_n)$, once for each higher value of $n$ the length of the interval $[t_{n-1},t_n)$ should be smaller what implies that the growth of function $z_n^\prime(t)$ should be each time bigger. However proving this assertion to general Banach spaces is a hard task.

To prove the following result, we also need to suppose that $t_1+t_2>1$. 

\begin{proposition} If $X$ is a Hilbert space, then property $(P)$ holds.
\end{proposition}

\begin{proof} Let $X$ be a Hilbert space. By definition of $cD_t^\alpha u(t)$ we observe
$$\Big\|cD_t^\alpha u(t)\Big|_{t=\tau_{n}}\Big\|^2_X=\sum_{k=2}^n{\Big[cD_t^\alpha z_k(t)\Big|_{t=\tau_{n}}\Big]^2},$$
for each $n\geq2$.

Now by assuming that property $(P)$ does not hold, there should exists $M>0$ such that
\begin{equation}\label{finalsec01}\Big\|cD_t^\alpha u(t)\Big|_{t=\tau_{n}}\Big\|_X\leq M,\quad\textrm{ for each }n\in\mathbb{N}.\end{equation}

Thus, for $n>2$ and by Remark \ref{prop}, it holds that
\begin{multline}\label{finalsec02}
cD_t^\alpha z_k(t)\Big|_{t=\tau_{n}}=\displaystyle\dfrac{1}{\Gamma(1-\alpha)}\int_{0}^{\tau_{n}}{(\tau_{n}-s)^{-\alpha}z_k^\prime(s)}\,ds\\=
\displaystyle\dfrac{1}{\Gamma(1-\alpha)}\left[\int_{\frac{t_{k-1}+t_{k}}{2}}^{t_{k}}{(\tau_{n}-s)^{-\alpha}z_k^\prime(s)}\,ds
+\displaystyle\int_{t_{k+1}}^{\frac{t_{k+1}+t_{k+2}}{2}}{(\tau_{n}-s)^{-\alpha}z_k^\prime(s)}\,ds\right],
\end{multline}
for each $k\in[2,n-1]$. Now by recalling that $z^\prime_k(t)\geq0$ in $[(t_{k-1}+t_k)/2,t_{k}]$ and $z^\prime_k(t)\leq0$ in $[t_{k+1},(t_{k+1}+t_{k+2})/2]$, we deduce
\begin{multline}\label{finalsec03}\int_{\frac{t_{k-1}+t_{k}}{2}}^{t_{k}}{(\tau_{n}-s)^{-\alpha}z_k^\prime(s)}\,ds\\\geq
\left(\tau_{n}-\frac{t_{k-1}+t_{k}}{2}\right)^{-\alpha}\int_{\frac{t_{k-1}+t_{k}}{2}}^{t_{k}}{z_k^\prime(s)}\,ds= \left(\tau_{n}-\frac{t_{k-1}+t_{k}}{2}\right)^{-\alpha}
\end{multline}
and
\begin{multline}\label{finalsec04}\int_{t_{k+1}}^{\frac{t_{k+1}+t_{k+2}}{2}}{(\tau_{n}-s)^{-\alpha}z_k^\prime(s)}\,ds\\
\geq\left(\tau_{n}-\frac{\tau_{n}}{2}\right)^{-\alpha}\int_{t_{k+1}}^{\frac{t_{k+1}+t_{k+2}}{2}}{z_k^\prime(s)}\,ds
= -\left(\frac{\tau_{n}}{2}\right)^{-\alpha}.\end{multline}

Therefore, by \eqref{finalsec01}, \eqref{finalsec02}, \eqref{finalsec03} and \eqref{finalsec04}, we achieve the inequality
$$\dfrac{1}{\Gamma(1-\alpha)}\left[\left(\tau_{n}-\frac{t_{k-1}+t_{k}}{2}\right)^{-\alpha}-\left(\frac{\tau_{n}}{2}\right)^{-\alpha}\right]\leq \Big\|cD_t^\alpha u(t)\Big|_{t=\tau_{n}}\Big\|_X\leq M,$$
since $t_{k-1}+t_k\geq t_1+t_2>1$.

Finally by taking the limit when $n\rightarrow\infty$ in both sides of the above inequality we obtain
$$\dfrac{1}{\Gamma(1-\alpha)}\left[\left(1-\frac{(t_{k-1}+t_{k})}{2}\right)^{-\alpha}-\left(\frac{1}{2}\right)^{-\alpha}\right]\leq M,$$
for any $k\geq 2$. However, since $(t_{k-1}+t_{k})/2\rightarrow 1$ when $k\rightarrow\infty$, the value on the left side of the above inequality would be greater than $M$ for $k$ sufficiently large, which is a contradiction. This completes the proof of the proposition.
\end{proof}

Last result allow us to establish the following theorem.

\begin{theorem}[Sharpness of ``Blow Up'' Conditions Revisited] Let $\alpha\in(0,1)$ and assume that $(P)$ holds in $X$. Then there exists $f_\alpha:\mathbb{R}^+\times X\rightarrow X$ continuous and locally Lipschitz which does not map every bounded set into bounded set, such that
  \begin{equation}\tag{FCP}
   \left\{\begin{array}{l}
   cD^\alpha_tu(t)=f_\alpha(t,u(t)),\,\,\,t> 0\\
   u(0)=v_1\in X,
   \end{array} \right.
  \end{equation}
 where $cD_t^\alpha$ is the Caputo's fractional derivative of order $\alpha$ and $v_1$ is the first element of the Schauder basis defined in section $3$, posses a bounded maximal solution $u:[0,1)\rightarrow X$.
\end{theorem}

\begin{proof} Here we discuss just the bounded set that is mapped by $f_\alpha$ into an unbounded set. Consider the sequence $\{(s_l,y_l)\}_{l=1}^\infty\subset[0,\infty)\times X$ such that
$$s_l=\tau_{n_l}\textrm{ and }y_l=\sum_{k=1}^{l+1}{(1/k^2)v_k},\textrm{ for }l\geq1,$$
where $\tau_{n_l}$ is given in hypothesis \textrm{(P)}.

It is not difficult to notice that $\{(s_l,y_l)\}_{l=1}^\infty$ is bounded, however
$$\|f_\alpha(s_l,y_l)\|_X=|\phi(H(s_l,y_l))|\,\left\|cD_t^\alpha u(t)\Big|_{t=\tau_{n_l}}\right\|_X$$
and since
$$|\phi(H(s_l,y_l))|=1,$$
the sequence $\{f_\alpha(s_l,y_l)\}_{l=1}^\infty$ is unbounded.\end{proof}

\section*{acknowledgements}
 Both authors would like to thank Universidade Federal do Esp\'{\i}rito Santo and Universidade Federal de Santa Catarina for the hospitality and support during respective short term visits.

\end{document}